\documentclass{article}
\usepackage[utf8]{inputenc}
\usepackage{babel}
\usepackage{amsmath,amsthm,amssymb}
\usepackage[shortlabels]{enumitem}
\usepackage{hyperref,cleveref}
\usepackage{accents}
\usepackage{xcolor}

\usepackage{fullpage}

\newcommand{\bivec}[1]{\accentset{\leftrightarrow}{#1}}

\title{Complete minors in digraphs with given dichromatic number}
\author
{
Tam\'{a}s M\'{e}sz\'{a}ros\thanks{Institut für Mathematik, Freie Universit\"{a}t Berlin, Germany. Research supported by Deutsche Forschungsgemeinschaft (DFG, German Research Foundation) under Germany’s Excellence Strategy - The Berlin Mathematics Research Center MATH+ (EXC-2046/1, project ID: 390685689). Email: \texttt{tmeszaros87@gmail.com}.}
\and
Raphael Steiner\thanks{Institut für Mathematik, Technische Universit\"at Berlin, Germany. Funded by DFG-GRK 2434 Facets of Complexity. Email: \texttt{steiner@math.tu-berlin.de}.}
}
\date{\today}

\newtheorem{thm}{Theorem}
\newtheorem{corollary}{Corollary}
\newtheorem{problem}{Problem}
\newtheorem{prop}{Proposition}

\begin{document}

\maketitle

\begin{abstract}
The \emph{dichromatic number} $\vec{\chi}(D)$ of a digraph $D$ is the smallest $k$ for which it admits a $k$-coloring where every color class induces an acyclic subgraph. 
Inspired by Hadwiger's conjecture for undirected graphs, se\-ve\-ral groups of authors have recently studied the containment of directed graph minors in digraphs with given dichromatic number. In this short note we improve several of the existing bounds 
and prove almost linear bounds by reducing the problem to a recent result of Postle on Hadwiger's conjecture.

\end{abstract}

\section{Introduction}


For a given integer $t\geq 1$ let $m_\chi(t)$ be the least integer for which it is true that every graph with chromatic number at least $m_\chi(t)$ contains a $K_t$-minor. Hadwiger's conjecture~\cite{H43}, which is one of the most important open problems in graph theory, states that $m_\chi(t)=t$ for all $t \ge 1$. The conjecture remains unsolved for $t \ge 7$. For many years, the best general upper bound on $m_\chi(t)$ was due to Kostochka~\cite{K82,K84} and Thomason~\cite{T84}, who independently proved that every graph of average degree at least $O(t\sqrt{\log t})$ contains a $K_t$-minor, implying that $m_\chi(t)=O(t\sqrt{\log t})$. Recently, however, there has been progress. First, Norine, Postle and Song~\cite{NPS19} showed that $m_\chi(t)=O\left(t(\log t)^\beta\right)$ (for any $\beta > \frac{1}{4}$), and then this was further improved by Postle~\cite{P20} to give $m_\chi(t)=O\left(t(\log\log t)^{6}\right)$. For more details about Hadwiger's conjecture the interested reader may consult the recent survey of Seymour~\cite{S16}.

This famous conjecture has influenced many researchers and different va\-ri\-a\-tions of it have been studied in various frameworks, one of which is directed graphs. In this case there are multiple ways to define a minor. Here we consider three popular variants: \emph{strong minors}, \emph{butterfly minors} and \emph{topological minors}. The containment of these different minors in dense digraphs as well as their relation to the dichromatic number have already been studied in several previous works, see e.g. \cite{AGSW20, KS15, J96} for strong minors, \cite{G20, KK15,J01,MSW19} for butterfly minors and \cite{A19, GSSz20, G21, M85, M95, M96, S00} for topological minors. 

\smallskip

Given digraphs $D$ and $H$, we say that $D$ is a \emph{strong $H$-minor model} if $V(D)$ can be partitioned into non-empty sets $\{X_v : v\in V(H)\}$ (called \emph{branch sets}) such that the digraph induced by $X_v$ is strongly-connected for all $v\in H$; and for every arc $(u,v)$ in $H$ there is an arc in $D$ from $X_u$ to $X_v$. More generally, we also say that $D$ \emph{contains $H$ as a strong minor} and write $D \succcurlyeq_s H$ if a subdigraph of $D$ is a strong $H$-minor model. Pause to note that strong minor containment defines a transitive relation on digraphs, that is, if $D_1 \succcurlyeq_s D_2$ and $D_2 \succcurlyeq_s D_3$ for digraphs $D_1,D_2,D_3$, then $D_1\succcurlyeq_s D_3$.

Given an undirected graph $G$ we denote by $\bivec{G}$ the directed graph with the same vertex set and for every edge $uv\in E(G)$ the vertices $u$ and $v$ are connected in $\bivec{G}$ by an arc in each direction. 
We will be particularly interested in forcing strong $\bivec{K}_t$-minors, as those also yield a strong $H$-minor for every digraph $H$ on at most $t$ vertices. Analogously to the undirected case, one can ask how large the dichromatic number of a digraph should be to guarantee that it contains a strong $\bivec{K}_t$ minor.
More precisely, we consider the function $sm_{\vec{\chi}}(t)$, which is the least integer for which it is true that every digraph $D$ with $\vec{\chi}(D) \ge sm_{\vec{\chi}}(t)$ satisfies $D\succcurlyeq_s \bivec{K}_t$. In a recent work, Axenovich, Gir\~{a}o, Snyder and Weber~\cite{AGSW20} showed that $sm_{\vec{\chi}}(t)$ exists for every $t \ge 1$ and proved the bounds
\begin{equation*}
    t+1\leq sm_{\vec{\chi}}(t)\leq t4^t.
\end{equation*}
Here we improve their upper bound substantially by reducing the problem to the undirected setting. 

\begin{thm}\label{thm:cor}
For every $t\geq 1$ we have 
\begin{equation*}
 sm_{\vec{\chi}}(t)\le 2m_\chi(t)-1.    
\end{equation*}
\end{thm}

By combining \Cref{thm:cor} with the aforementioned result of Postle we get that $sm_{\vec{\chi}}(t)= O\left(t(\log\log t)^{6}\right)$.   
%

\smallskip  

Now let us turn to butterfly minors. Given a digraph $D$ and an arc $(u,v) \in A(D)$, this arc is called \emph{(butterfly-)contractible} if $v$ is the only out-neighbor of $u$ or if $u$ is the only in-neighbor of $v$ in $D$. Given such a contractible arc $e$, the digraph $D/e$ 
is obtained from $D$ by merging $u$ and $v$ into a common vertex and joining their in- and out-neighborhoods, ignoring parallel arcs.
A \emph{butterfly minor} of a digraph $D$ is any digraph that can be obtained by repeatedly deleting arcs, deleting vertices or contracting arcs.

In \cite{MSW19}, inspired by Hadwiger's conjecture, Millani, Steiner and Wiederrecht raised the question that for a given integer $k\geq 1$, what is the largest butterfly minor closed class $\mathcal{D}_k$ of $k$-colorable digraphs, and they gave a precise characterization of $\mathcal{D}_2$ as \emph{non-even digraphs}.
The question concerning a characterization of $\mathcal{D}_k$ for $k \ge 3$ is closely related to the question of forcing complete butterfly minors in digraphs. For an integer $t \ge 1$, let us define $bm_{\vec{\chi}}(t)$ as the least integer such that every digraph $D$ with $\vec{\chi}(D) \ge bm_{\vec{\chi}}(t)$ contains $\bivec{K}_t$ as a butterfly minor, and put 
\begin{equation*}
    b(x):=\max\left\{t \ge 1\ \mid \ bm_{\vec{\chi}}(t) \le x\right\}
\end{equation*}
for the integer inverse function of $bm_{\vec{\chi}}(\cdot)$. Let us further denote by $\mathcal{K}_t$ the class of all digraphs with no $\bivec{K}_t$ as a butterfly minor. Then, on the one hand, every digraph excluding $\bivec{K}_{b(k+1)}$ as a butterfly minor is colourable with $bm_{\vec{\chi}}(b(k+1))-1 \le k$ colours. On the other hand, every digraph in $\mathcal{D}_k$ must exclude $\bivec{K}_{k+1}$ as a butterfly minor, since its dichromatic number exceeds $k$. Therefore, for every $k$ we have 
\begin{equation*}
    \mathcal{K}_{b(k+1)} \subseteq  \mathcal{D}_k  \subseteq \mathcal{K}_{k+1}.
\end{equation*}
To see how tight the the above inclusions are one needs to obtain good lower bounds on $b(k+1)$, or equivalently good upper bounds on $bm_{\vec{\chi}}(t)$. In this direction, as an app\-li\-cation of \Cref{thm:cor} we prove the following corollary. 



\begin{corollary}\label{cor:but}
For $t \ge 1$ we have $bm_{\vec{\chi}}(t) \le 2m_\chi(2t)-1= O(t(\log\log t)^{6})$.
\end{corollary}

For the sake of completeness we remark that a lower bound of $t+1\le bm_{\vec{\chi}}(t)$ follows by taking $D=\bivec{G}$ where $G$ is the complete graph on $t+2$ vertices with a $5$-cycle removed. It is a simple exercise to verify that $\bivec{\chi}(D)=t$ but it contains no butterfly $\bivec{K}_t$-minor.

\smallskip

Finally, we consider topological minors. Given a digraph $H$, a \emph{subdivision of $H$} is any digraph obtained by replacing every arc $(u,v)\in A(H)$ by a directed path from $u$ to $v$, such that subdivision-paths of different arcs are internally vertex-disjoint. Then $H$ is said to be a \emph{topological minor} of some digraph $D$ if $D$ contains a subdivision of $H$ as a subgraph.

Aboulker, Cohen, Havet, Lochet, Moura and Thomassé~\cite{A19} initiated the study of the existence of various subdivisions in digraphs of large dichromatic number. For a digraph $H$ they introduced the parameter $\text{mader}_{\vec{\chi}}(H)$, the \emph{dichromatic Mader number of $H$}, as the least integer such that any digraph $D$ with $\vec{\chi}(D)\geq \text{mader}_{\vec{\chi}}(H)$ contains a subdivision of $H$. In their main result they proved that if $H$ is a digraph with $n$ vertices and $m$ arcs, then 
\begin{equation*}
    n\leq \text{mader}_{\vec{\chi}}(H)\leq 4^m(n-1)+1.
\end{equation*}
Gishboliner, Steiner and Szabó~\cite{GSSz20} conjectured that $\text{mader}_{\vec{\chi}}(\bivec{K}_t)\le Ct^2$ for some absolute constant $C$, however, it seems surprisingly hard to find a polynomial upper bound even for quite simple digraphs $H$. An indication for this increased difficulty compared to the undirected case could be that for digraphs it is not even possible to force a $\bivec{K}_3$-subdivision by means of large minimum out- and in-degree (compare~\cite{M85}).
In \cite{GSSz20} the authors still managed to identify a wide class of graphs, called octus graphs\footnote{We note that this class, in particular, includes orientations of cactus graphs (and hence orientations of cycles), as well as bioriented forests.}, for which the lower bound is tight. Their result means that given a digraph $D$ with $\vec{\chi}(D)\geq n$ it contains the subdivision of every octus graph on at most $n$ vertices.

Here, along the same line of thinking, as a corollary of \Cref{thm:cor} we prove a similar result for another class of digraphs. By slightly abusing the terminology, we call a digraph $D$ \emph{subcubic} if $D$ is an orientation of a graph with maximum degree at most three such that the in- and out-degree of any vertex is at most two.

\begin{corollary}\label{cor:top}
For $n\ge 1$ if $D$ is a digraph with $\vec{\chi}(D) \ge 22n$ then it contains a subdivision of every subcubic digraph on at most $n$ vertices.
\end{corollary}



\paragraph{Notation.}

For a digraph $D$ and a set $S\subseteq V(D)$ we denote by $D[S]$ the subdigraph spanned by the vertices in $S$. The set $S$ is called \emph{acyclic} if $D[S]$ is an acyclic digraph.
We call $D$ \emph{strongly-connected} if for every ordered pair $u,v$ of vertices in $D$ there is a directed path in $D$ from $u$ to $v$.  
An in-/out-arborescence is a rooted directed tree where every arc is directed towards/away from the root. For the starting/ending point of an arc we will also use the names tail/head.

A \emph{(proper) coloring} of an undirected graph $G$ with colors in a set $A$ is a map $f:V(G)\rightarrow A$ where neighbouring vertices are mapped to different colors, or equivalently $f^{-1}(a)$ is an independent set for every $a\in A$. If $|A|=k$ then $f$ is called a \emph{ $k$-coloring}. Analogously, an \emph{(acyclic) $k$-coloring} of a digraph $D$ is a map $f:V(D)\rightarrow A$ with $|A|=k$ where $f^{-1}(a)$ is an acyclic set for every $a\in A$. The minimum $k$ for which a $k$-coloring exists is the \emph{chromatic} (resp. \emph{dichromatic}) \emph{number} of the undirected graph $G$ (resp. digraph $D$), which we shall denote by $\chi(G)$ (resp. $\vec{\chi}(D))$.

\section{Proofs}

\subsection{Strong minors}


The proof of \Cref{thm:cor} will be based on the following result.

\begin{thm}\label{thm:main}
For every digraph $D$ there is an undirected graph $G$ such that 
\begin{enumerate}[(i)]
\item $D$ is a strong $\bivec{G}$-minor model, and 
\item $\vec{\chi}(D)\leq 2\chi(G)$.
\end{enumerate}
\end{thm}

\begin{proof}
To start with, let us first fix a partition $X_1,X_2,\dots, X_m$ of $V(D)$ such that for every $i\in \{1,2,\dots,m\}$ the set $X_i$ is an inclusion-wise maximal subset of $V(D)\setminus \left(X_1\cup \cdots \cup X_{i-1}\right)$ with $D[X_i]$ strongly connected and $\vec{\chi}(D[X_i]) \le 2$. Note that the $X_i$'s are well-defined since the one vertex-digraph is strongly connected and $2$-colorable. Now we define $G$ to be the undirected simple graph with vertex set $\{X_1,\dots,X_m\}$ and $X_iX_j\in E(G)$ if and only if there are arcs in both directions between $X_i$ and $X_j$ in $D$. Then, by definition, $D$ is a strong $\bivec{G}$-minor model, as one can simply take $X_1,X_2,\dots,X_m$ as the branch sets. 

Therefore, what remains to prove is property (ii). For this let us assume that $\chi(G)=k$ and fix a proper coloring $f_G:V(G)\rightarrow \{c_1,c_2,....,c_k\}$ of $G$. Now, for every $i$ take an arbitrary acyclic two-coloring of $D[X_i]$ (which exists by assumption) with colors $\{c_i',c_i''\}$. The rest of the proof is about showing that by putting these colorings together we obtain an acyclic coloring $f_D$ of $D$ with the $2k$ colors $\{c_{1}',c_{1}'',c_{2}',c_{2}'',\dots,c_{k}',c_{k}''\}$. 

Assume for contradiction that this is not the case, and there is a directed cycle $C$ in $D$ which is monochromatic. We may, without loss of generality, assume that $C$ is a shortest such cycle, in particular, it is and induced cycle. Let $i_0$ be the smallest index for which $C$ contains a vertex from $X_{i_0}$. Note that, in particular, $V(C)\subseteq V(D)\setminus \left(X_1\cup \cdots \cup X_{i_0-1}\right)$ and, as $f_D$ is a proper coloring on $D[X_{i_0}]$, the cycle $C$ cannot be fully contained in $X_{i_0}$. Hence, $C$ contains a subsequence $u,w_1,\dots,w_\ell,v$ of consecutive vertices on $C$ with $(u,w_1),(w_1,w_2),\dots,(w_\ell,v) \in A(C)$, such that $u,v\in X_{i_0}$ (possibly $u=v$), $w_1,\ldots,w_\ell \in X_{i_0+1} \cup \cdots \cup X_m$, and $\ell>0$. 

Let $s \in \{1,\ldots,\ell\}$ be the smallest index such that $w_{s}$ has an out-neighbour in $X_{i_0}$, and denote this out-neighbor by $x \in X_{i_0}$. We claim that $w_{s}$ has no in-neighbor in $D$ that is contained in $X_{i_0}$. Suppose towards a contradiction that there exists $y \in X_{i_0}$ such that $(y,w_{s}) \in A(D)$. Let $j>{i_0}$ be such that $w_s\in X_j$. Then, because of the arcs $(y,w_{s}) , (w_s,x) \in A(D)$, we have $X_{i_0}X_j\in E(G)$ and hence $f_G(X_{i_0})\neq f_G(X_j)$. This in turn implies that $f_D(u)\neq f_D(w_s)$ and $f_D(v)\neq f_D(w_s)$ which contradicts the monochromaticity of $C$. Hence, we may assume that $w_s$ has no in-neighbor contained in $X_{i_0}$. In particular, this implies $s \ge 2$. Let us now consider the set
\begin{equation*}
 X=X_{i_0}\cup\{w_1,\dots,w_{s}\}\subseteq V(D)\setminus \left(X_1\cup \cdots \cup X_{i_0-1}\right).   
\end{equation*}
 It is clearly strongly connected, as $X_{i_0}$ is so and $u,w_1,\ldots,w_{s},x$ induce a directed path (or cycle in case $u=x$) starting and ending in $X_{i_0}$. Moreover, any extension of an acyclic $\{1,2\}$-coloring of $D[X_{i_0}]$ to a $\{1,2\}$-coloring of $D[X]$ where $w_1,\ldots,w_{s-1}$ receive color $1$ and $w_{s}$ receives color $2$ is acyclic. Indeed, by the definition of $s$, there are no arcs starting in $\{w_1,\dots, w_{s-1}\}$ and ending in $X_{i_0}$, and by the inducedness of $C$ there are no arcs spanned between non-consecutive vertices inside $\{w_1,\dots, w_{s-1}\}$. Adding the fact that $w_s$ has no in-neighbours in $X_{i_0}$, these imply that any directed cycle in $D[X]$ is either fully contained in $D[X_{i_0}]$, or contains both $w_s$ and at least one vertex in $\{w_1,\dots, w_{s-1}\}$. In any case, it is not monochromatic. However, the existence of the set $X$ then contradicts with the maximality of $X_{i_0}$, which finishes the proof.
\end{proof}

Now we can easily deduce \Cref{thm:cor} from \Cref{thm:main}.

\begin{proof}[Proof of  \Cref{thm:cor}.]
Let $D$ be a digraph with $\vec{\chi}(D) \ge 2m_\chi(t)-1$. By Theorem~\ref{thm:main} there exists an undirected graph $G$ such that $\vec{\chi}(D) \le 2\chi(G)$ and $D \succcurlyeq_s \bivec{G}$. This implies that $\chi(G) \ge m_\chi(t)$, and hence $G$ contains a $K_t$-minor. Taking the same branch sets in $\bivec{G}$ which give a $K_t$-minor in $G$ shows that $\bivec{G} \succcurlyeq_s \bivec{K}_t$, and by transitivity $D\succcurlyeq_s \bivec{K}_t$. Since $D$ was arbitrarily chosen such that $\vec{\chi}(D) \ge 2m_\chi(t)-1$, this proves that $sm_{\vec{\chi}}(t) \le 2m_\chi(t)-1$, as required.
\end{proof}

\subsection{Butterfly minors}

\Cref{cor:but} follows directly from Theorem~\ref{thm:cor} and the following proposition.


\begin{prop}
Every strong $\bivec{K}_{2t}$-minor model contains $\bivec{K}_t$ as a butterfly minor.
\end{prop}
\begin{proof}
 Let $D$ be a strong $\bivec{K}_{2t}$-minor model and let $\{X_1^+, X_1^-,\ldots,X_t^+,X_t^-\}$ be a corres\-ponding partition of $V(D)$ into $2t$ 
 branch sets. 
 In particular, for every $i \in \{1,\ldots,t\}$ there exist $r_i^+ \in X_i^+$ and $r_i^- \in X_i^-$ such that $(r_i^-,r_i^+) \in A(D)$. Since $D[X_i^-]$ and $D[X_i^+]$ are strongly connected digraphs,  there exist\footnote{Such trees can easily be obtained by considering a breadth-first in-search (resp. out-search) starting from $r_i^-$ (resp. $r_i^+$).} oriented spanning trees $T_i^- \subseteq D[X_i^-]$ and  $T_i^+ \subseteq D[X_i^+]$ such that $T_i^-$ is an in-arbores\-cence rooted at $r_i^-$ and $T_i^+$ is an out-arborescence rooted at $r_i^+$. Let us consider the spanning subdigraph $D'$ of $D$ consisting of the arcs contained in 
 \begin{equation*}
     T:=\bigcup_{i=1}^{t}\Big({\{(r_i^-,r_i^+)\} \cup A(T_i^+) \cup A(T_i^-)}\Big),
 \end{equation*}
 as well as all arcs of $D$ starting in $X_i^+$ and ending in $X_j^-$ for 
 $i \neq j$. Then every arc of $D'$ contained in $T$ is either the unique arc in $D'$ emanating from its tail or the unique arc in $D'$ entering its head. It follows that all arcs in $T$ are butterfly-contractible. Note that the contraction of an arc does not affect the butterfly-contractibility of other arcs, hence the digraph $D'/T$, obtained from $D'$ by successively contracting all arcs in $T$, is a butterfly minor of $D$. The vertices of $D'/T$ can be labelled $v_1,\ldots,v_t$, where $v_i$ denotes the vertex corresponding to the contraction of the (weakly) connected component of $D'$ inside $X_i^+ \cup X_i^-$. 
 As $D$ is a strong $\bivec{K}_{2t}$-minor model, by definition of $D'$ for every $(i,j) \in \{1,\ldots,k\}^2$ with $i \neq j$, there exists an arc in $D'$ starting in $X_i^+$ and ending in $X_j^-$. Therefore, $D'/T$ is a butterfly minor of $D$ isomorphic to $\bivec{K}_t$, concluding the proof.
\end{proof}

\subsection{Topological minors}

Finally, we prove \Cref{cor:top}.

\begin{proof}[Proof of \Cref{cor:top}.]
As a first step note that given $n \in \mathbb{N}$, every undirected graph $G$ with minimum degree at least $10.5n>n+6.291\cdot\frac{3}{2}n$ contains every $n$-vertex subcubic graph as a minor. This follows directly from a result of Reed and Wood~\cite{reedwood}, who proved that every graph with average degree at least $n+6.291m$ contains every graph with $n$ vertices and $m$ edges as a minor. 


Let now $D$ be any digraph with $\vec{\chi}(D) \ge 22n$, $F$ a subcubic digraph on $n \ge 2$ vertices and $H$ its underlying undirected subcubic graph.  By Theorem~\ref{thm:main} there exists an undirected graph $G$ such that $D$ is a strong $\bivec{G}$-minor model and $\chi(G) \ge 11n$. In particular, $G$ contains a subgraph of minimum degree at least $11n-1>10.5n$ and hence, by our earlier remark, an $H$-minor. This implies that $\bivec{G}$ contains a strong $\bivec{H}$-minor and hence $D$ does so.
However, as $F \subseteq \bivec{H}$, it also follows that $D$ contains a strong $F$-minor, i.e. a subdigraph $D'$ which is a strong $F$-minor model. Let $\{X_f : f \in V(F)\}$ be a branch set partition of $V(D')$ witnessing this.
Recall that, by definition, for every arc $e=(u_1,u_2) \in A(F)$ there exist vertices $v(e,u_1) \in X_{u_1}$ and  $v(e,u_2) \in X_{u_2}$ such that $\left(v(e,u_1),v(e,u_2)\right) \in A(D')\subseteq A(D)$. 

Let next $u \in V(F)$ be an arbitrary vertex with total degree $d=d(u)\in \{0,1,2,3\}$ and let us denote the arcs incident to $u$ by $e_1,\dots,e_d$. Furthermore, for $i=1,\ldots,d$ we put $v_i:=v(e_i,u)$. We claim that there exists a vertex $b(u) \in X_u$ and for every $i=1,\ldots,d$ a directed path $P_i^u$ in $D[X_u]$ such that
\begin{itemize}
    \item $P_1^u,\dots,P_d^u$ are internally vertex-disjoint;
    \item if $u$ is the tail of $e_i$, then  $P_i^u$ is a directed path from $b(u)$ to $v_i$;
    \item  if $u$ is the head of $e_i$, then $P_i^u$ is a directed path from $v_i$ to $b(u)$.
\end{itemize}
This claim holds trivially if $d=0$, and if $d=1$ then we can simply put $b(u)=v_1$ and let $P_1^u$ be the trivial one-vertex path consisting of $v_1$. 

If $d=2$ then, without loss of generality, by the symmetry of reversing all arcs in $D$ and $F$, we may assume that $u$ is the head of $e_1$. We then can put $b(u):=v_1$, let $P_1^u$ be the trivial one-vertex path consisting of $v_1$, and take $P_2^u$ to be any directed path in in $D[X_u]$ from $v_1$ to $v_2$, which exists by strong connectivity.

Finally suppose $d=3$. Since $F$ is subcubic, $u$ either has in-degree one and out-degree two, or vice versa. As before, without loss of generality, by symmetry we may assume that the first case occurs, and it is $e_1$ that enters $u$ and $e_2$ and $e_3$ that emanate from it. Take now $P_{12}$ and $P_{13}$ to be directed paths in $D[X_u]$ starting at $v_1$ and ending at $v_2$ and $v_3$, respectively. We define now $b(u)$ as the first vertex in $V(P_{12})$ that we meet when traversing $P_{13}$ backwards (starting at $v_3$); $P_1^u$ as the subpath of $P_{12}$ directed from $v_1$ to $b(u)$; $P_2^u$ as the subpath of $P_{12}$ directed from $b(u)$ to $v_2$; and $P_{3}^3$ as the subpath of $P_{13}$ directed from $b(u)$ to $v_3$. It follows by definition that $P_1, P_2, P_3$ are internally vertex-disjoint, and hence the claim follows.

To finish the proof, let $S \subseteq D$ be a subdigraph with vertex set
\begin{equation*}
    V(S):=\bigcup_{u \in V(F)}\left(\bigcup_{i=1}^{d(u)}V(P_i^u)\right),
\end{equation*}
and arcs
\begin{equation*}
    A(S):=\left\{\Big(v(e,u_1),v(e,u_2)\Big)\ \Big|\ e=(u_1,u_2) \in A(F)\right\} \cup \left(\bigcup_{u \in V(F)}\left(\bigcup_{i=1}^{d(u)}A(P_i^u)\right)\right).
\end{equation*}
$S$ is a digraph isomorphic to a subdivision of $F$ in which a vertex $u \in V(F)$ is represented by the branch-vertex $b(u)$. This concludes the proof.
\end{proof}

\section{Concluding remarks}
In this note we showed that $sm_{\vec{\chi}}(t) \le 2m_\chi(t)-1$ and $bm_{\vec{\chi}}(t)  \le 2m_\chi(2t)-1$ for any $t \ge 1$. As far as lower bounds are considered, it is not hard to see that $m_\chi(t) \le \min\{sm_{\vec{\chi}}(t),bm_{\vec{\chi}}(t)\}$ for every $t \ge 1$. Indeed, for any graph $G$ with $\chi(G) \ge \min\{sm_{\vec{\chi}}(t),bm_{\vec{\chi}}(t)\}$, as $\vec{\chi}(\bivec{G})=\chi(G)$, by definition $\bivec{G}$ contains $\bivec{K}_t$ either as a strong minor or as a butterfly minor, each of which implies that $G$ contains a $K_t$-minor.  
Therefore, our results reduce the question about the asymptotics of $sm_{\vec{\chi}}(t)$ and $bm_{\vec{\chi}}(t)$ to the well-studied undirected version of the problem. Also, as Hadwiger's conjecture is known to be true for small values, for $3 \le t \le 6$ we have
$$t+1 \le sm_{\vec{\chi}}(t) \le 2t-1\quad\text{and}\quad t+1 \le bm_{\vec{\chi}}(t) \le 4t-1.$$
We believe that the upper bounds should not be tight. 
To support this intuition, let us mention that a more careful analysis of our proof of Theorem~\ref{thm:cor} yields the stronger statement that any digraph $D$ with $\vec{\chi}(D) \ge 2m_\chi(t)-1$ contains a strong $\bivec{K}_t$-minor model in which between any two branch sets, there are at least two arcs spanned in both directions. Under the assumption that Hadwiger's conjecture is true, the bound $2t-1$ for this stronger property would be sharp, as shown by $\bivec{K}_{2t-2}$. 
This indicates that our proof should not be expected to give a tight bound for the problem of forcing a strong $\bivec{K}_t$-minor. Instead it seems plausible that $sm_{\vec{\chi}}(t)=t+1$ (and maybe $bm_{\vec{\chi}}(t)=t+1$) for any $t \ge 3$.

\begin{problem}
Does every digraph $D$ with $\vec{\chi}(D) \ge t+1$ contain $\bivec{K}_t$ as a strong minor (butterfly minor)?
\end{problem}

Already resolving the first open case $t=3$ would be quite interesting.


\section*{Acknowledgments} We would like to thank Patrick Morris for fruitful discussions on the topic.

\end{document}